\documentclass[12pt]{article}

\usepackage[margin=1.25in]{geometry}

\usepackage{amssymb,amsmath}
\usepackage{amsthm}
\usepackage{hyperref}
\usepackage{mathrsfs}
\usepackage{textcomp}
\usepackage{enumerate}
\usepackage{framed}
\usepackage{ulem}
\usepackage{color,xcolor}

\usepackage{soul}
\soulregister\cite7 
\soulregister\citep7 
\soulregister\citet7 
\soulregister\ref7 
\soulregister\pageref7 
\soulregister\eqref7

\hypersetup{
    colorlinks,%
    citecolor=blue,%
    filecolor=blue,%
    linkcolor=blue,%
    urlcolor=black
}

\newtheorem{theorem}{Theorem}[section]
\newtheorem{definition}[theorem]{Definition}
\newtheorem{lemma}[theorem]{Lemma}

\newtheorem{proposition}[theorem]{Proposition}
\newtheorem{corollary}[theorem]{Corollary}

\numberwithin{equation}{section}

\makeatletter

\newdimen\bibspace
\renewenvironment{thebibliography}[1]{%
 \section*{\refname 
       \@mkboth{\MakeUppercase\refname}{\MakeUppercase\refname}}%
     \list{\@biblabel{\@arabic\c@enumiv}}%
          {\settowidth\labelwidth{\@biblabel{#1}}%
           \leftmargin\labelwidth
           \advance\leftmargin\labelsep
           \itemsep\bibspace
           \parsep\z@skip     %
           \@openbib@code
           \usecounter{enumiv}%
           \let\p@enumiv\@empty
           \renewcommand\theenumiv{\@arabic\c@enumiv}}%
     \sloppy\clubpenalty4000\widowpenalty4000%
     \sfcode`\.\@m}
    {\def\@noitemerr
      {\@latex@warning{Empty `thebibliography' environment}}%
     \endlist}

\makeatother

           \newcommand{\ud}{\mathrm{d}}
\newcommand{\be}{\begin{equation}}      \newcommand{\ee}{\end{equation}}

\begin{document}

\title{New existence results for prescribed fractional $Q$-curvatures problem on $\mathbb{S}^n$ under pinching conditions}

\author{Zhongwei Tang\thanks{Z. Tang is supported by National Natural Science Foundation of China (12071036).},\, Ning Zhou}

\date{}

\maketitle

\begin{abstract}
In this paper we study the prescribed fractional $Q$-curvatures problem of order $2 \sigma$ on the $n$-dimensional standard sphere $(\mathbb{S}^{n}, g_0)$, where $n\geq3$, $\sigma\in(0,\frac{n-2}{2})$. By combining critical points at infinity approach with Morse theory we obtain new existence results under suitable pinching conditions.
\end{abstract}


\section{Introduction}

The classical Nirenberg problem asks: which function $K$ on the standard $n$-dimensional sphere $(\mathbb{S}^{n}, g_0)$ is the scalar curvature (Gauss curvature in dimension $n=2$) of a metric $g$ that is conformal to $g_0$? On $\mathbb{S}^{2}$, setting $g=e^{2 u} g_0$, the Nirenberg problem is equivalent to solving the following nonlinear elliptic equation
\begin{equation}\label{Gauss}
-\Delta_{g_0} u+1=K e^{2 u}\quad \text { on }\, \mathbb{S}^{2},
\end{equation}
where $\Delta_{g_0}$ is the Laplace-Beltrami operator on $(\mathbb{S}^n, g_0)$. On $\mathbb{S}^{n}$ $(n \geq 3)$, writing the conformal metric as $g=u^{\frac{4}{n-2}} g_0$, the Nirenberg problem is equivalent to the existence of the following nonlinear elliptic equation involving the Sobolev critical exponent
\begin{equation}\label{NP}
-\Delta_{g_0} u+\frac{n(n-2)}{4} u=K u^{\frac{n+2}{n-2}}, \quad u>0 \quad \text { on }\, \mathbb{S}^{n}.
\end{equation}

First of all, Eq. \eqref{NP} is not always solvable. Indeed, we have the Kazdan-Warner type obstruction: for any conformal Killing vector field $X$ on $\mathbb{S}^{n}$, there holds
$$
\int_{\mathbb{S}^{n}}(\nabla_{X} K) u^{\frac{2 n}{n-2}} \,\ud v_{g_{0}}=0
$$
for any solution $u$ of \eqref{NP}. Hence, if $K(\xi)=\xi_{n+1}+2$ for example, then Eq. \eqref{NP} has no positive solutions. Many works were devoted to the problem trying to understand under what conditions on $K$ problem \eqref{Gauss} and \eqref{NP} are solvable, which are divided into three main categories:
\begin{enumerate}[(1)]
  \item Group invariance conditions. See  Escobar-Schoen \cite{EscobarSchoenConformal1986}, Chen \cite{ChenScalar1989}, Hebey \cite{HebeyChangements1990}, etc.
  \item Mountain Path type condition. See Chen-Ding \cite{ChenDingScalar1987}, etc.
  \item Bahri-Coron type condition. See Bahri-Coron \cite{BahriCoronThe1991}, Chang-Yang \cite{ChangYangPrescribing1987,ChangYangConformal1988,ChangYangA1991}, etc.
\end{enumerate}

This paper is concerned with the problem of prescribing fractional $Q$-curvature of order $2 \sigma$, $0<\sigma<\frac{n-2}{2}$, on $(\mathbb{S}^{n},g_0)$. The problem consists of finding a new metric $g$ on $\mathbb{S}^{n}$, conformally equivalent to $g_{0}$ with prescribed fractional $Q$-curvature. Let $K\geq0$ be a smooth function on $\mathbb{S}^{n}$. Set $g=u^{\frac{4}{n-2 \sigma}} g_{0}$, where $u>0$ is a smooth function on $\mathbb{S}^n$. Then $K$ is the fractional $Q$-curvature of order $2 \sigma$ of the metric $g$ if and only if $u$ is a solution to the equation
\begin{equation}\label{1.1}
P_{\sigma} u=c(n, \sigma) K u^{\frac{n+2 \sigma}{n-2 \sigma}}, \quad u>0\quad \text { on }\, \mathbb{S}^{n},
\end{equation}
where $n\geq2$, $0<\sigma<\frac{n}{2}$, $c(n, \sigma)=\Gamma(\frac{n}{2}+\sigma) / \Gamma(\frac{n}{2}-\sigma)$, $\Gamma$ is the Gamma function and $P_{\sigma}$ is the $2 \sigma$-order conformal Laplacian on $\mathbb{S}^{n} .$ The operator $P_{\sigma}$ can be uniquely expressed as following
$$
P_{\sigma}=\frac{\Gamma(B+\frac{1}{2}+\sigma)}{\Gamma(B+\frac{1}{2}-\sigma)}, \quad B=\sqrt{-\Delta_{g_{0}}+\Big(\frac{n-1}{2}\Big)^{2}}.
$$
Moreover, $P_{\sigma}$ can be seen as the pull back operator of $(-\Delta)^{\sigma}$ on $\mathbb{R}^{n}$ via the stereographic projection, where $(-\Delta)^{\sigma}$ is the fractional Laplacian operator. Let $\mathcal{N}$ be the north pole of $\mathbb{S}^{n}$ and define
\be\label{phi}
\begin{aligned}
\Phi: \mathbb{R}^{n} &\rightarrow \mathbb{S}^{n} \backslash\{\mathcal{N}\},\\
x &\mapsto\Big(\frac{2 x}{1+|x|^{2}}, \frac{|x|^{2}-1}{|x|^{2}+1}\Big)
\end{aligned}
\ee
be the inverse of stereographic projection operator from $\mathbb{S}^{n} \backslash\{\mathcal{N}\}$ to $\mathbb{R}^{n}$. Then, by the conformal invariance of $P_{\sigma}$, one has the following relation
$$
P_{\sigma}(\varphi) \circ \Phi=|J_{\Phi}|^{-\frac{n+2\sigma}{2 n}}(-\Delta)^{\sigma}(|J_{\Phi}|^{\frac{n-2\sigma}{2 n}}(\varphi \circ \Phi)),\quad \forall\, \varphi \in C^{\infty}(\mathbb{S}^{n}),
$$
where $|J_{\Phi}|=(\frac{2}{1+|x|^{2}})^{n} .$ Then, for a solution $u$ to \eqref{1.1}, $v(x)=|J_{\Phi}|^{\frac{n-2 \sigma}{2 n}} u(\Phi(x))$ satisfies
$$
(-\Delta)^{\sigma} v=c(n, \sigma)(K \circ \Phi) v^{\frac{n+2 \sigma}{n-2 \sigma}},\quad v>0 \quad \text { on }\, \mathbb{R}^{n}.
$$

Denote $H^{\sigma}(\mathbb{S}^{n})$ as the $\sigma$-order fractional Sobolev space that consists of all functions $u \in L^{2}(\mathbb{S}^{n})$ such that $(1-\Delta_{g_0})^{\sigma / 2} u \in L^{2}(\mathbb{S}^{n})$, with the norm
$$
\|u\|=\Big(\int_{\mathbb{S}^{n}} u P_{\sigma} u  \,\ud v_{g_{0}}\Big)^{1/2}.
$$
Beckner \cite{BecknerSharp1993} showed that the Yamabe ratio
$$
S_{n}:=\inf _{u \in H^{\sigma}(\mathbb{S}^{n}) \backslash\{0\}} \frac{\int_{\mathbb{S}^{n}} u P_{\sigma} u \,\ud v_{g_{0}}}{(\int_{\mathbb{S}^{n}} u^{\frac{2n}{n-2\sigma}} \,\ud v_{g_{0}})^{\frac{n-2\sigma}{n}}}=\frac{\omega_{n}^{\frac{2 \sigma}{n}} \Gamma(\frac{n}{2}+\sigma)}{\Gamma(\frac{n}{2}-\sigma)},
$$
where $\omega_{n}$ denotes the area of the $n$-dimensional unit sphere.

The problem of prescribing fractional $Q$-curvature of order $2\sigma$ on $\mathbb{S}^{n}$ can be considered as generalizations and extensions of the Nirenberg problem and the prescribed Paneitz-Branson curvature problem (see for example \cite{DjadliMalchiodiAhmedouI2002}, \cite{DjadliMalchiodiAhmedouII2002}, and \cite{ChtiouiRiganOn2011}). One may see the work of, among many others, Escobar \cite{EscobarConformal1996}, Chang-Xu-Yang \cite{ChangXuYangA1998}, Han-Li \cite{HanLiThe1999}, Djadli-Malchiodi-Ahmedou \cite{DjadliMalchiodiAhmedouThe2004}, Abdelhedi-Chtioui \cite{AbdelhediChtiouiOn2013} for $\sigma=1 / 2 ;$ Jin-Li-Xiong \cite{JinLiXiongOn2014,JinLiXiongOn2015}, Chen-Liu-Zheng \cite{ChenLiuZhengExistence2016}, Abdelhedi-Chtioui-Hajaiej \cite{AbdelhediChtiouiHajaiejA2016} for $\sigma \in(0,1)$; Jin-Li-Xiong \cite{JinLiXiongThe2017} for $\sigma\in (0,n/2)$; Wei-Xu \cite{WeiXuOn1998,WeiXuPrescribing2009}, Brendle \cite{BrendlePrescribing2003} for $\sigma=n / 2$, and Zhu \cite{ZhuPrescribing2016} for $\sigma>n / 2$.

Problem \eqref{1.1} has a natural variational structure. Solutions can be found as critical points (up to a multiplicative constant) of the functional
$$
J_K(u)=\frac{1}{(\int_{\mathbb{S}^{n}} K u^{\frac{2 n}{n-2 \sigma}} \,\ud v_{g_{0}})^{\frac{n-2 \sigma}{n}}}, \quad u \in \Sigma,
$$
where $\Sigma$ is the unit sphere of $H^{\sigma}(\mathbb{S}^{n})$.

Since $\frac{2 n}{n-2 \sigma}$ corresponds to the critical exponent of the fractional Sobolev embeddings $H^{\sigma}(\mathbb{S}^{n}) \hookrightarrow L^{q}(\mathbb{S}^{n})$, the functional $J_K$ fails to satisfy the Palais-Smale condition on $\Sigma^{+}:=\{u \in \Sigma \mid u>0\}$. This constitutes a strong obstruction for the application of the direct methods of the calculus of variations or even standard variational methods. Therefore, more refined techniques are needed as ``Critical points at infinity theory'' introduced by A. Bahri which we will follow in this work. The critical points at infinity are the ends of the noncompact flow-lines of the gradient vector field $-J_K^{\prime}$ and the precise definition will be introduced by Definition \ref{def:3.1}.

Recently, Malchiodi and Mayer \cite{MalchiodiMayerPinching2019} obtained an interesting existence criterion of the Nirenberg problem \eqref{NP} under some pinching condition. More precisely, let $n \geq 5$ and $K \in C^{\infty}(\mathbb{S}^{n})$ be a positive Morse function, under the following pinching condition
$$
\frac{K_{\max }}{K_{\min }} \leq\Big(\frac{3}{2}\Big)^{\frac{1}{n-2}},
$$
where $K_{\max }:=\max _{\mathbb{S}^{n}} K$, $K_{\min }:=\min _{\mathbb{S}^{n}} K$, and $K$ has at least two critical points with negative Laplacian, they were able to prove that \eqref{NP} has at least a solution.

The analog of this result for the Nirenberg problem on standard half spheres $\mathbb{S}_{+}^{n}$ with Neumann condition was proved in Ahmedou and Ben Ayed \cite{AhmedouBenThe2021}. Very recently, Fourti \cite{FourtiNew2021} yielded similar results for the prescribed mean curvature problem on unit ball $\mathbb{B}^n$ with boundary $\mathbb{S}^{n-1}$.

The aim of this paper is to extend these kind of results to the prescribing fractional $Q$-curvature problem of order $2\sigma$ on $\mathbb{S}^{n}$.

Our main assumption for the function $K$ is the so-called non-degeneracy condition:

${\bf (nd)}$ We assume that $K>0$ is a $C^{2}(\mathbb{S}^{n})$ function and for each critical point $y$ of $K$ we have $\Delta_{g_0} K(y) \neq 0$. That is, $K$ is a positive Morse function and non-degenerate on $\mathbb{S}^{n}$.

Note that by Sard-Smale Theorem, the set of functions having only non-degenerate critical points with $\Delta_{g_0} K(y) \neq 0$ is dense in the set of $C^{2}$ functions. Therefore it is easy to find examples of functions satisfying our assumption.

Let
$$
\mathcal{K}:=\{y \in \mathbb{S}^{n} \mid \nabla_{g_0} K(y)=0\},\quad \mathcal{K}^{+}:=\{y \in \mathcal{K} \mid -\Delta_{g_0} K(y)>0\},
$$
and
$$
\mathcal{M}:=\{\tau_{p}=(y_{1}, \cdots, y_{p}) \in(\mathcal{K}^{+})^{p},\, p\geq 1 \mid y_{i} \neq y_{j},\, \forall\, 1 \leq i \neq j \leq p\}.
$$
By Morse Lemma, $\mathcal{K}^{+}$ is a finite set since $\mathbb{S}^{n}$ is compact.

Our first result provides a new and easily verifiable criterion for the existence of solutions to \eqref{1.1}:

\begin{theorem}\label{thm:1.2}
Let $n \geq 3$, $\sigma\in(0,\frac{n-2}{2})$ and $K$ satisfying the assumption ${\bf (nd)}$. If the following conditions hold
\begin{enumerate}[(i)]
  \item
$$
\frac{K_{\max}}{K_{\min}}<\Big(\frac{3}{2}\Big)^{\frac{\sigma}{n-2 \sigma}};
$$
  \item
  $$
  \sharp\mathcal{K}^+ \geq 2,
  $$
where $\sharp\mathcal{A}$ denotes the cardinality of the finite set $\mathcal{A}$.
\end{enumerate}
Then the problem \eqref{1.1} has at least one solution.
\end{theorem}

The above pinching condition (i) of Theorem \ref{thm:1.2} can be relaxed when combined with some counting index formula. Namely, we prove

\begin{theorem}\label{thm:1.4}
Let $n \geq 3$, $\sigma\in(0,\frac{n-2}{2})$ and $K$ satisfying the assumption ${\bf (nd)}$. If the following conditions hold
\begin{enumerate}[(i)]
  \item
$$
\frac{K_{\max}}{K_{\min}}<2^{\frac{\sigma}{n-2 \sigma}};
$$
  \item
$$
A_{1}:=\sum_{z \in \mathcal{K}^+}(-1)^{n-i n d(K, z)} \neq 1,
$$
where ${ind}(K, z)$ denotes the Morse index of $K$ at $z$.
\end{enumerate}
Then the problem \eqref{1.1} has at least one solution.
\end{theorem}

Our approach follows some arguments developed in \cite{AhmedouBenThe2021} based on the techniques related to the critical points at infinity theory combined with Morse theory. Firstly, we describe the lack of compactness of the problem and characterize the critical points at infinity of its associated functional $J_{K}$. Then we compute the topological contribution of the critical points at infinity to the difference of topology between the level sets of the functional $J_{K}$. Finally, we will derive our existence results by means of two ideas in \cite{AhmedouBenThe2021} related to the pinching condition. The first one is that the pinching condition means that suitable sublevels of $J_{K}$ are contractible. The second one is that critical levels at infinity of $J_{K}$ stratify depending on the number of bubbles.

The structure of our paper is the following. In Section 2, we recall the variational framework and review the lack of compactness. In Section 3, we characterize the critical points at infinity and calculate their topological contributions. Theorems \ref{thm:1.2} and \ref{thm:1.4} are proved in Section 4.

\section{The lack of compactness}

In this section we set up the variational framework of the problem \eqref{1.1} and recall the description of its lack of compactness.

Note that equation \eqref{1.1} admits a natural variational characterization, the Euler-Lagrange functional is
$$
J_{K}(u)=\frac{\|u\|^2}{(\int_{\mathbb{S}^{n}} K u^{\frac{2 n}{n-2 \sigma}} \,\ud v_{g_{0}})^{\frac{n-2 \sigma}{n}}},\quad u \in H^{\sigma}(\mathbb{S}^{n}).
$$
If $u$ is a critical point of the functional $J_K$ in $\Sigma^{+}$, then up to a multiplicative constant, $u$ is a solution of \eqref{1.1}. However, since $\frac{2 n}{n-2 \sigma}$ corresponds to the critical exponent of the fractional Sobolev embeddings $H^{\sigma}(\mathbb{S}^{n}) \hookrightarrow L^{q}(\mathbb{S}^{n})$, the functional $J_K$ does not satisfy the Palais-Smale condition on $\Sigma^{+}$. More precisely, this leads to the possibility of existence of the critical points at infinity, which are the limits of noncompact orbits for the gradient flow of $-J_K$. In fact, let $s \mapsto \eta(s, u)$, $u \in \Sigma^{+}$ be a flow line of the gradient flow of $-J_K$. If $\frac{2 n}{n-2 \sigma}$ in $J_K$ is replaced by $\frac{2 n}{n-2 \sigma}-\varepsilon$, $\varepsilon>0$, $\eta(s, u)$ converges to a critical point in $\Sigma^{+}$. However, in the critical case $\frac{2 n}{n-2 \sigma}$, there are possible obstacles to finding critical points of $J_K$: these are the so called critical points at infinity.

To describe non-converging Palais-Smale sequences we introduce the following notation.

For $a \in \mathbb{S}^{n}$ and $\lambda>0$, we define on $\mathbb{S}^n$ the standard bubble to be
$$
\delta_{a, \lambda}(x)=\bar{c} \frac{\lambda^{\frac{n-2 \sigma}{2}}}{(1+\frac{\lambda^{2}-1}{2}(1-\cos d(x, a)))^{\frac{n-2 \sigma}{2}}},
$$
where $d$ is the geodesic distance on $\mathbb{S}^{n}$ and $\bar{c}$ is chosen such that $\delta_{a, \lambda}$ satisfies
$$
P_{\sigma} \delta_{a, \lambda}=\delta_{a, \lambda}^{\frac{n+2 \sigma}{n-2 \sigma}} \quad \text { on }\, \mathbb{S}^{n}.
$$

We define now the set of potential critical points at infinity associated to the functional $J_{K}$. For $p \in \mathbb{N}_+$ and $\varepsilon>0$, let
$$
\begin{aligned}
V(p, \varepsilon):=\Big\{&u \in \Sigma \mid \exists\, \alpha_{1}, \cdots, \alpha_{p}>0,\, \exists\, a_{1}, \cdots, a_{p} \in \mathbb{S}^{n},\,
\exists\, \lambda_{1}, \cdots, \lambda_{p}>\varepsilon^{-1} \text { with }\\
&\|u-\sum_{i=1}^{p} \alpha_{i} \delta_{a_{i}, \lambda_{i}}\|<\varepsilon,\, |J_K(u)^{\frac{n}{n-2 \sigma}} \alpha_{i}^{\frac{4 \sigma}{n-2 \sigma}} K(a_{i})-1|<\varepsilon,\, \forall\, i,\, \varepsilon_{i j}<\varepsilon,\, \forall\, i \neq j\Big\},
\end{aligned}
$$
where
$$
\varepsilon_{i j}=\Big(\frac{\lambda_{i}}{\lambda_{j}}+\frac{\lambda_{j}}{\lambda_{i}}+\lambda_{i} \lambda_{j} d(a_{i}, a_{j})^{2}\Big)^{-\frac{n-2\sigma}{2}}.
$$

In the following we describe non-converging Palais-Smale sequences. Such a description follows from concentration-compactness arguments, see \cite{BahriAn1996} for details.

\begin{proposition}\label{pro:2.1}
Assume that $J_K$ has no critical points in $\Sigma^{+}$. Let $\{u_{k}\}\subset \Sigma^{+}$ be a sequence such that $J_{K}(u_{k})$ is bounded and $J_K^{\prime}(u_{k}) \rightarrow 0$ as $k\to\infty$. Then there exists an integer $p \in \mathbb{N}_+$, a positive sequence $\varepsilon_{k} \rightarrow 0$, and an extracted subsequence of $\{u_{k}\}$, still denoted $\{u_{k}\}$, such that $u_{k} \in V(p, \varepsilon_{k}) .$
\end{proposition}

If $u$ is a function in $V(p, \varepsilon)$, one can find an optimal representation, following the ideas introduced in \cite{BahriCritical1989}. Namely, we have

\begin{proposition}\label{pro:2.2}
Let $p \in \mathbb{N}_+$ and $\varepsilon>0$ small enough. For any $u \in V(p, \varepsilon)$, the following minimization problem
\be\label{MiniPro}
\min _{\alpha_{i}>0,\, a_{i} \in \mathbb{S}^{n},\, \lambda_{i}>0}\Big\|u-\sum_{i=1}^{p} \alpha_{i} \delta_{a_{i}, \lambda_{i}}\Big\|
\ee
has a unique solution $(\alpha, a, \lambda)$ up to a permutation. Thus, we can write $u$ as follows
$$
u=\sum_{i=1}^{p} \alpha_{i} \delta_{a_{i}, \lambda_{i}}+v,
$$
where $v$ belongs to $H^{\sigma}(\mathbb{S}^{n})$ and satisfies the following condition:
$$
\langle v, \varphi_i\rangle=0 \quad \text { for } \,\varphi_i =\delta_{a_{i}, \lambda_{i}},\, \frac{\partial \delta_{a_{i}, \lambda_{i}}}{\partial a_{i}},\, \frac{\partial \delta_{a_{i}, \lambda_{i}}}{\partial \lambda_{i}}, \quad i=1, \cdots, p,\eqno(V_0)
$$
here, $\langle \cdot,\cdot \rangle$ denotes the scalar product in $H^{\sigma}(\mathbb{S}^{n})$ defined by
$$
\langle u, v\rangle=\int_{\mathbb{S}^{n}} v P_{\sigma} u \,\ud v_{g_{0}}.
$$
\end{proposition}

In the next, we denote $v \in(V_{0})$ to say that $v$ satisfies $(V_{0})$. We first give an expansion in $V(p, \varepsilon)$ of the functional $J_{K}$ on functions of the parameters $\alpha_{i}, a_{i}, \lambda_{i}, v$.

\begin{proposition}\label{pro:expansionJK}
If $p\in \mathbb{N}_+$, $\varepsilon>0$ small enough and $u=\sum_{i=1}^{p} \alpha_{i} \delta_{a_{i}, \lambda_{i}}+v \in V(p, \varepsilon)$ with $v$ satisfies $(V_0)$, we have
$$
\begin{aligned}
J_K(u)=& \frac{\sum_{i=1}^{p} \alpha_{i}^{2} S_n}{(\sum_{i=1}^{p} \alpha_{i}^{\frac{2 n}{n-2 \sigma}} K(a_{i}) S_n)^{\frac{n-2 \sigma}{n}}}\Big[1-\frac{n-2 \sigma}{n} \frac{c_{2}}{\Gamma_{1}} \sum_{i=1}^{p} \alpha_{i}^{\frac{2 n}{n-2 \sigma}} \frac{\Delta_{g_0} K(a_{i})}{\lambda_{i}^{2}}\Big] \\
&+\frac{\sum_{i=1}^{p} \alpha_{i}^{2} S_n}{(\sum_{i=1}^{p} \alpha_{i}^{\frac{2 n}{n-2 \sigma}} K(a_{i}) S_n)^{\frac{n-2 \sigma}{n}}}\Big[\sum_{i \neq j} c_{0}^{\frac{2 n}{n-2 \sigma}} c_{1} \omega_{n} \varepsilon_{i j}\Big(\frac{\alpha_{i} \alpha_{j}}{\Gamma_{2}}-\frac{2 \alpha_{i}^{\frac{n+2 \sigma}{n-2 \sigma}} \alpha_{j} K(a_{i})}{\Gamma_{1}}\Big)\Big] \\
&+\frac{\sum_{i=1}^{p} \alpha_{i}^{2} S_n}{(\sum_{i=1}^{p} \alpha_{i}^{\frac{2 n}{n-2 \sigma}} K(a_{i}) S_n)^{\frac{n-2 \sigma}{n}}}\Big[f(v)+Q(v, v)+o\Big(\sum_{i \neq j} \varepsilon_{i j}\Big)+o(\|v\|^{2})\Big],
\end{aligned}
$$
where
$$
f(v)=-\frac{2}{\Gamma_{1}} \int_{\mathbb{S}^{n}} K\Big(\sum_{i=1}^{p} \alpha_{i} \delta_{a_{i}, \lambda_{i}}\Big)^{\frac{n+2 \sigma}{n-2 \sigma}} v \,\ud v_{g_{0}},
$$
$$
Q(v, v)=\frac{1}{\Gamma_{2}}\|v\|^{2}-\frac{2 n+4 \sigma}{\Gamma_{1}(n-2 \sigma)} \sum_{i=1}^{p} \int_{\mathbb{S}^{n}} K(\alpha_{i} \delta_{a_{i}, \lambda_{i}})^{\frac{4 \sigma}{n-2 \sigma}} v^{2} \,\ud v_{g_{0}},
$$
$$
\Gamma_{1}=\sum_{i=1}^{p} \alpha_{i}^{\frac{2 n}{n-2 \sigma}} K(a_{i}) S_n, \quad \Gamma_{2}=\sum_{i=1}^{p} \alpha_{i}^{2} S_n,
$$
and $c_0, c_1, c_2$ are some constants.
\end{proposition}

\begin{proof}
The proof is similar to Appendix A of \cite{ChenLiuZhengExistence2016} and we omit it here.
\end{proof}

Set
$$
H_{\varepsilon}(a, \lambda)=\{v \in H^{\sigma}(\mathbb{S}^{n}) \mid v \text { satisfies } (V_0) \text{ and } \|v\| \leq \varepsilon\}.
$$
The following result gets rid of the $v$-contributions, i.e., $v$ can be neglected with respect to the concentration phenomenon.

\begin{proposition}\label{pro:2.4}
There is a $C^{1}$-map which to each $(\alpha_{i}, a_{i}, \lambda_{i})$ such that $u=\sum_{i=1}^{p} \alpha_{i} \delta_{a_{i}, \lambda_{i}}$ belongs to $V(p, \varepsilon)$ associates $\bar{v}=\bar{v}(\alpha, a, \lambda)$ such that $\bar{v}$ is unique and satisfies
$$
J_K\Big(\sum_{i=1}^{p} \alpha_{i} \delta_{a_{i}, \lambda_{i}}+\bar{v}\Big)=\min_{v \in(V_{0})}J_K\Big(\sum_{i=1}^{p} \alpha_{i} \delta_{a_{i}, \lambda_{i}}+v\Big).
$$
Moreover, we have the following estimates:
$$
\|\bar{v}\| \leq C\Big[\sum_{i=1}^{p}\Big(\frac{|\nabla_{g_0} K(a_{i})|}{\lambda_{i}}+\frac{1}{\lambda_{i}^{2}}\Big)+\sum_{i \neq j}
\begin{cases}
\varepsilon_{i j}^{\frac{n+2 \sigma}{2(n-2 \sigma)}} (\log \varepsilon_{i j}^{-1})^{\frac{n+2 \sigma}{2 n}} &\text { if }\, n\geq 6\sigma \\
\varepsilon_{i j} (\log \varepsilon_{i j}^{-1})^{\frac{n-2 \sigma}{n}} & \text { if }\, n <6\sigma
\end{cases}\Big].
$$
\end{proposition}

\begin{proof}
By Proposition \ref{pro:2.2}, the parameterization of $V(p, \varepsilon)$ is given by
$$
\begin{aligned}
B_{\varepsilon} \times H_{\varepsilon}(a, \lambda) &\rightarrow V(p, \varepsilon) \\
(\alpha, a, \lambda, v) &\mapsto u=\sum_{i=1}^{p} \alpha_{i} \delta_{a_{i}, \lambda_{i}}+v,
\end{aligned}
$$
where $B_{\varepsilon}=\{(\alpha, a, \lambda) \in (\mathbb{R}_+)^{p} \times(\mathbb{S}^{n})^{p} \times(\mathbb{R}_+)^{p} \mid \varepsilon_{i j} \leq \varepsilon,\, \lambda_{i}>{\varepsilon}^{-1}\}$, $(\alpha, a, \lambda)$ is the solution of the minimizing problem \eqref{MiniPro} in $B_{\varepsilon}$, $v \in H_{\varepsilon}(a, \lambda)$. Since $(\alpha, a, \lambda) \in B_{\varepsilon}$, $\varepsilon_{i j}$'s are small. Using the same arguments as in \cite{AbdelhediChtiouiThe2007}, we have
$$
Q(v, v) \geq \delta_{0}\|v\|^{2}, \quad \forall\, v \in H_{\varepsilon}(a, \lambda),
$$
where $\delta_{0}>0$ is a constant. Thus there exists an invertible operator $A$ such that $Q(v,v)=\frac{1}{2}\langle A v, v\rangle$ on $H_{\varepsilon}(a, \lambda)$ and $\beta_{0} \mathrm{Id} \leq A \leq \beta_{1} \mathrm{Id}$, where $\beta_{1}>\beta_{0}>0$ are some constants. By Proposition \ref{pro:expansionJK}, we have
$$
\begin{aligned}
&J_K\Big(\sum_{i=1}^{p} \alpha_{i} \delta_{a_{i}, \lambda_{i}}+v\Big) \\
=&\frac{\sum_{i=1}^{p} \alpha_{i}^{2} S_n}{(\sum_{i=1}^{p} \alpha_{i}^{\frac{2 n}{n-2 \sigma}} K(a_{i}) S_n)^{\frac{n-2 \sigma}{n}}}\Big[1-\frac{n-2 \sigma}{n} \frac{c_{2}}{\Gamma_{1}} \sum_{i=1}^{p} \alpha_{i}^{\frac{2 n}{n-2 \sigma}} \frac{\Delta_{g_0} K(a_{i})}{\lambda_{i}^{2}}\Big] \\
&+\frac{\sum_{i=1}^{p} \alpha_{i}^{2} S_n}{(\sum_{i=1}^{p} \alpha_{i}^{\frac{2 n}{n-2 \sigma}} K(a_{i}) S_n)^{\frac{n-2 \sigma}{n}}}\Big[\sum_{i \neq j} c_{0}^{\frac{2 n}{n-2 \sigma}} c_{1} \omega_{n} \varepsilon_{i j}\Big(\frac{\alpha_{i} \alpha_{j}}{\Gamma_{2}}-\frac{2 \alpha_{i}^{\frac{n+2 \sigma}{n-2 \sigma}} \alpha_{j} K(a_{i})}{\Gamma_{1}}\Big)\Big] \\
&+\frac{\sum_{i=1}^{p} \alpha_{i}^{2} S_n}{(\sum_{i=1}^{p} \alpha_{i}^{\frac{2 n}{n-2 \sigma}} K(a_{i}) S_n)^{\frac{n-2 \sigma}{n}}}\Big[f(v)+\frac{1}{2}\langle A v, v\rangle+o\Big(\sum_{i \neq j} \varepsilon_{i j}\Big)+o(\|v\|^{2})\Big].
\end{aligned}
$$
Since the term $o(\|v\|^{2})$ is twice differentiable in $v$, and it's differential at the origin is $o(\|v\|)$, we get
$$
\langle J_K^{\prime}(v), h\rangle=\frac{\sum_{i=1}^{p} \alpha_{i}^{2} S_n}{(\sum_{i=1}^{p} \alpha_{i}^{\frac{2 n}{n-2 \sigma}} K(a_{i}) S_n)^{\frac{n-2 \sigma}{n}}}[f(h)+\langle A v, h\rangle+\langle o(\|v\|), h\rangle] .
$$
Note that the second differential of $o(\|v\|^{2})$ is $o(1)$, it follows that the functional $f(v)+\frac{1}{2}\langle A v, v\rangle+o(\|v\|^{2})$ is coercive in a neighborhood of the origin. Consequently, $f(v)+\frac{1}{2}\langle A v, v\rangle+o(\|v\|^{2})$ has a unique minimum $\bar{v}$ in a neighborhood of 0 in $H_{\varepsilon}(a, \lambda)$, and $\bar{v}$ satisfies
$$
f+A \bar{v}+o(\|\bar{v}\|)=0.
$$
Since the operator $A+o(1)$ is positive and invertible in a neighborhood of the origin, we obtain that $A^{-1}$ satisfies $\frac{1}{2 \beta_{1}} \mathrm{Id} \leq A^{-1} \leq \frac{2}{\beta_{0}} \mathrm{Id}$. Moreover,
$$
\|\bar{v}\| \leq C_1\|A^{-1} f\| \leq C_2\|f\|,
$$
where $C_1, C_2>0$ are some constants and $f$ is defined in Proposition \ref{pro:expansionJK}. Thus, we only need to estimate $\|f\|$. First, we have
$$
\begin{aligned}
&\int_{\mathbb{S}^{n}} K\Big(\sum_{i=1}^{p} \alpha_{i} \delta_{a_{i}, \lambda_{i}}\Big)^{\frac{n+2 \sigma}{n-2 \sigma}} v \,\ud v_{g_{0}}\\
=&\sum_{i=1}^{p} \alpha_{i}^{\frac{n+2 \sigma}{n-2 \sigma}} \int_{\mathbb{S}^{n}} K \delta_{a_{i}, \lambda_{i}}^{\frac{n+2 \sigma}{n-2 \sigma}} v \,\ud v_{g_{0}}\\
&+O(\|v\|)\Big[\int_{\mathbb{S}^{n}} \sum_{i \neq j}(\alpha_{i} \delta_{a_{i}, \lambda_{i}})^{\frac{8 n \sigma}{n^{2}-4 \sigma^{2}}} \inf [(\alpha_{i} \delta_{a_{i}, \lambda_{i}})^{\frac{2 n}{n+2 \sigma}},(\alpha_{j} \delta_{a_{j}, \lambda_{j}})^{\frac{2 n}{n+2 \sigma}}] \,\ud v_{g_{0}}\Big]^{\frac{n+2 \sigma}{2 n}}.
\end{aligned}
$$
Easy computations lead to
$$
\begin{aligned}
&\Big[\int_{\mathbb{S}^{n}} \sum_{i \neq j}(\alpha_{i} \delta_{a_{i}, \lambda_{i}})^{\frac{8 n \sigma}{n^{2}-4 \sigma^{2}}} \inf [(\alpha_{i} \delta_{a_{i}, \lambda_{i}})^{\frac{2 n}{n+2 \sigma}},(\alpha_{j} \delta_{a_{j}, \lambda_{j}})^{\frac{2 n}{n+2 \sigma}}] \,\ud v_{g_{0}}\Big]^{\frac{n+2 \sigma}{2 n}}\\
=&
\begin{cases}
O(\sum_{i \neq j}(\varepsilon_{i j})^{\frac{n+2 \sigma}{2(n-2 \sigma)}}(\log \varepsilon_{i j}^{-1})^{\frac{n+2 \sigma}{2 n}})\quad &\text{ if }\, n \geq 6\sigma,\\
O(\sum_{i \neq j}\varepsilon_{i j}(\log \varepsilon_{i j}^{-1})^{\frac{n-2 \sigma}{n}})\quad &\text{ if }\, n <6\sigma.
\end{cases}
\end{aligned}
$$
Using the fact that $\langle v, \delta_{a_i, \lambda_i}\rangle=0$, we obtain
$$
\begin{aligned}
\int_{\mathbb{S}^{n}} K \delta_{a_{i}, \lambda_{i}}^{\frac{n+2 \sigma}{n-2 \sigma}} v \,\ud v_{g_{0}}=&
O\Big(\int_{\mathbb{R}^{n}}(|\nabla_{g_0} K(a_{i})||x-a_{i}|+|x-a_{i}|^{2}) \tilde{\delta}_{a_{i}, \lambda_{i}}^{\frac{n+2\sigma}{n-2\sigma}}\tilde{v}\,\ud x\Big)\\
=&O\Big(\int_{B_{\rho}(a_i)\cup B_{\rho}^c(a_i)}(|\nabla_{g_0} K(a_{i})||x-a_{i}|+|x-a_{i}|^{2}) \tilde{\delta}_{a_{i}, \lambda_{i}}^{\frac{n+2\sigma}{n-2\sigma}}\tilde{v}\,\ud x\Big)\\
=& O\Big(\Big(\frac{|\nabla_{g_0} K(a_{i})|}{\lambda_{i}}+\frac{1}{\lambda_{i}^{2}}\Big)\|v\|\Big),
\end{aligned}
$$
where $\tilde{\delta}_{a_{i}, \lambda_{i}}={\delta}_{a_{i}, \lambda_{i}}\circ \Phi$, $\tilde{v}=v\circ \Phi$, and $\Phi$ is defined by \eqref{phi}. The proof is complete.
\end{proof}

\section{Critical points at infinity and their topological contribution}

Following Bahri and Coron in \cite{BahriCritical1989,BahriAn1996,BahriCoronThe1991}, we will use the following definition later.

\begin{definition}\label{def:3.1}
A critical point at infinity of $J_{K}$ on $\Sigma^{+}$ is a limit of a flow line $u(s)$ of the equation
$$
\begin{cases}
\displaystyle\frac{\ud u(s)}{\ud s}=-J_{K}^{\prime}(u(s)), \\
u(0)=u_{0} \in H^{\sigma}(\mathbb{S}^{n}),
\end{cases}
$$
such that $u(s)$ remains in $V(p, \varepsilon(s))$ for $s \geq s_{0}$, where $\varepsilon(s)>0$ and $\rightarrow 0$ as $s \rightarrow+\infty$, $s_0>0$ is some constant, and $u_{0}$ is an initial value.
\end{definition}

Using Proposition \ref{pro:2.2}, $u(s)$ can be written as
$$
u(s)=\sum_{i=1}^{p} \alpha_{i}(s) \delta_{a_{i}(s), \lambda_{i}(s)}+v(s).
$$
Let $\alpha_{i}:=\lim _{s \rightarrow+\infty} \alpha_{i}(s)$ and $a_{i}:=\lim _{s \rightarrow+\infty} a_{i}(s)$, then such a critical point at infinity is denoted by
$$
\sum_{i=1}^{p} \alpha_{i} \delta_{a_{i}, \infty}\quad \text { or }\quad (a_{1}, \cdots, a_{p})_{\infty}.
$$

Let us first recall the characterization of the critical points at infinity of $J_K$ for all $\sigma \in(0, \frac{n-2}{2})$ from \cite{SharafChtiouiConformal2020}. The characterization can be obtained through the estimates of the gradient vector field $J^{\prime}_{K}$ and the expansion of $J_{K}$.

\begin{proposition}\label{pro:3.4}
Let $n \geq 3$, $\sigma\in (0,\frac{n-2}{2})$ and $K$ satisfying the assumption ${\bf (nd)}$. Assume that $J_K$ has no critical points in $\Sigma^{+}$. Then the critical points at infinity of $J_K$ are
$$
(y_{1}, \cdots, y_{p})_{\infty}=\sum_{i=1}^{p} \frac{1}{K(y_{i})^{\frac{n-2 \sigma}{n}}} \delta_{y_{i}, \infty},
$$
where $(y_{1}, \cdots, y_{p}) \in \mathcal{M}$.

Moreover, the Morse index of $(y_{1}, \cdots, y_{p})_{\infty}$ is
$$
i_{\infty}(y_{1}, \cdots, y_{p})_{\infty}=p-1+\sum_{i=1}^{p}(n-ind(K, y_{i})),
$$
where $ind(K, y_{i})$ denotes the Morse index of $K$ at $y_{i}$.
\end{proposition}

First, we have the following Morse Lemma, which completely gets rid of the $v$-contributions and shows that the functional $J_K$ behaves, at infinity, as $J_K(\sum_{i=1}^{p} \alpha_{i} \delta_{\tilde{a}_{i}, \tilde{\lambda}_{i}})+\|V\|^{2}$, where $V$ is a variable completely independent of $\tilde{a}_{i}, \tilde{\lambda}_{i}$. The proof is analog to the case $\sigma\in (0,1)$ in \cite{ChenLiuZhengExistence2016} and we omit here.

\begin{lemma}\label{lem:getridofv}
There is a covering $\{O_{l}\}$, a subset $\{(\alpha_{l}, a_{l}, \lambda_{l})\}$ of the base space for the bundle $V(p, \varepsilon)$ and a diffeomorphism $\xi_{l}: V(p, \varepsilon) \rightarrow V(p, \varepsilon^{\prime})$ for some $\varepsilon^{\prime}>0$ with
$$
\xi_{l}\Big(\sum_{i=1}^{p} \alpha_{i} \delta_{a_{i}, \lambda_{i}}+\bar{v}\Big)=\sum_{i=1}^{p} \alpha_{i} \delta_{\tilde{a}_{i}, \tilde{\lambda}_{i}},
$$
such that
$$
J_K\Big(\sum_{i=1}^{p} \alpha_{i} \delta_{a_{i}, \lambda_{i}}+v\Big)=J_K\Big(\sum_{i=1}^{p} \alpha_{i} \delta_{\tilde{a}_{i}, \tilde{\lambda}_{i}}\Big)+\frac{1}{2} J_K^{\prime \prime}\Big(\sum_{i=1}^{p} \alpha_{i} \delta_{a_{i}, \lambda_{i}}\Big) V_{l} \cdot V_{l},
$$
where $(\alpha, a, \lambda) \in O_{l}$, $(\alpha, \tilde{a}, \tilde{\lambda})$ is independent of $O_{l}$ and $V_{l}$ is orthogonal to $\delta_{\tilde{a}_{i}, \tilde{\lambda}_{i}}$, $\frac{\partial \delta_{\tilde{a}_{i}, \tilde{\lambda}_{i}}}{\partial \tilde{a}_{i}}$, $\frac{\partial \delta_{\tilde{a}_{i}, \tilde{\lambda}_{i}}}{\partial \tilde{\lambda}_{i}}$.
\end{lemma}

Now we introduce a Morse Lemma at Infinity of $J_{K}$ near its critical points at infinity.

\begin{lemma}\label{lem:3.5}
Let $n\geq3$, $\sigma\in(0,\frac{n-2}{2})$ and $u=\sum_{i=1}^{p} \alpha_{i} \delta_{a_i, \lambda_i}+v \in V(p, \varepsilon)$, $p\geq1$, such that $a_{i} \in B_{\rho}(y_i)$, $\forall\, i=1, \cdots, p$ and $(y_{1}, \cdots, y_{p}) \in \mathcal{M}$. Then there exists a change of variables such that
$$
J_{K}(u)=S_{n}^{\frac{2 \sigma}{n}}\Big(\sum_{i=1}^{p} \frac{1}{K(y_{i})^{\frac{n-2 \sigma}{2 \sigma}}}\Big)^{\frac{2 \sigma}{n}}\Big(1-|\widetilde{\alpha}|^{2}+\sum_{i=1}^{p}
(|a_{i}^{-}|^{2}-|a_{i}^{+}|^{2})+C\sum_{i=1}^{p} \frac{1}{\lambda_{i}^{2}}\Big)+\|V\|^{2},
$$
where $\widetilde{\alpha} \in \mathbb{R}^{p-1}$ and $(a_{i}^{+}, a_{i}^{-})$ are the coordinates of $a_{i}$ near $y_{{i}}$ along the stable and unstable manifold for $K$.
\end{lemma}

\begin{proof}
From Proposition \ref{pro:expansionJK} and Lemma \ref{lem:getridofv}, we have
$$
\begin{aligned}
J_K(u)=& \frac{\sum_{i=1}^{p} \alpha_{i}^{2} S_n}{(\sum_{i=1}^{p} \alpha_{i}^{\frac{2 n}{n-2 \sigma}} K(a_{i}) S_n)^{\frac{n-2 \sigma}{n}}}\Big[1-\frac{n-2 \sigma}{n} \frac{c_{2}}{\Gamma_{1}} \sum_{i=1}^{p} \alpha_{i}^{\frac{2 n}{n-2 \sigma}} \frac{\Delta_{g_0} K(a_{i})}{\lambda_{i}^{2}}\Big] \\
&+O\Big(\sum_{i \neq j} \varepsilon_{i j}\Big)+\frac{1}{2} J_K^{\prime \prime}\Big(\sum_{i=1}^{p} \alpha_{i} \delta_{a_{i}, \lambda_{i}}\Big) V_{l} \cdot V_{l}.
\end{aligned}
$$
Since $(y_{1}, \cdots, y_{p}) \in \mathcal{M}$, we get $|a_{i}-a_{j}|>0$. Therefore,
$$
\varepsilon_{i j} \sim \frac{1}{(\lambda_{i} \lambda_{j})^{\frac{n-2 \sigma}{2}}} \leq C\Big(\frac{1}{\lambda_{i}^{n-2 \sigma}}+\frac{1}{\lambda_{j}^{n-2 \sigma}}\Big)=o\Big(\frac{1}{\lambda_{i}^{2}}+\frac{1}{\lambda_{j}^{2}}\Big).
$$
Thus, the expansion of the functional $J_K$ can be rewritten as follows:
$$
\begin{aligned}
J_K(u)=&\frac{\sum_{i=1}^{p} \alpha_{i}^{2} S_n}{(\sum_{i=1}^{p} \alpha_{i}^{\frac{2 n}{n-2 \sigma}} K(a_{i}) S_n)^{\frac{n-2 \sigma}{n}}}\Big[1+\frac{n-2 \sigma}{n} \frac{c_{2}}{\Gamma_{1}} \sum_{i=1}^{p} \alpha_{i}^{\frac{2 n}{n-2 \sigma}} \frac{-\Delta_{g_0} K(a_{i})}{\lambda_{i}^{2}}+o\Big(\sum_{i=1}^{p} \frac{1}{\lambda_{i}^{2}}\Big)\Big]\\
&+\frac{1}{2} J_K^{\prime \prime}\Big(\sum_{i=1}^{p} \alpha_{i} \delta_{a_{i}, \lambda_{i}}\Big) V_{l} \cdot V_{l}.
\end{aligned}
$$
Except the term
$$
g(\alpha, a)=\frac{\sum_{i=1}^{p} \alpha_{i}^{2}S_n}{(\sum_{i=1}^{p} \alpha_{i}^{\frac{2n}{n-2\sigma}} K(a_{i})S_n)^{\frac{n-2\sigma}{n}}},
$$
all others are positive on the right hand side of the above equality. Note that $g(\alpha, a)$ is homogeneous in the variable $\alpha$ and has a maximum point
$$
\Big(\frac{1}{K(a_{1})^{\frac{n-2 \sigma}{4 \sigma}}}, \frac{1}{K(a_{2})^{\frac{n-2 \sigma}{4 \sigma}}}, \cdots, \frac{1}{K(a_{p})^{\frac{n-2 \sigma}{4 \sigma}}}\Big),
$$
thus the index of this critical point is $p-1$. On the other hand, $g(\alpha, a)$ has a single critical point $y=(y_{{1}}, y_{{2}}, \cdots, y_{{p}})$ in the $a$ variable. Thus, using the Morse Lemma in finite dimensional, after a change of variables, we have the following normal form,
$$
J_{K}(u)=S_{n}^{\frac{2 \sigma}{n}}\Big(\sum_{i=1}^{p} \frac{1}{K(y_{i})^{\frac{n-2 \sigma}{2 \sigma}}}\Big)^{\frac{2 \sigma}{n}}\Big(1-|\widetilde{\alpha}|^{2}+\sum_{i=1}^{p}
(|a_{i}^{-}|^{2}-|a_{i}^{+}|^{2})+C\sum_{i=1}^{p} \frac{1}{\lambda_{i}^{2}}\Big)+\|V\|^{2}.
$$
This completes the proof of the Lemma.
\end{proof}

Using Lemma \ref{lem:3.5}, we identify the level sets of critical points at infinity. Let us denote the corresponding critical point at infinity in Lemma \ref{lem:3.5} by $(y_{1}, \cdots, y_{p})_{\infty}$.

\begin{corollary}\label{cor:3.6}
The critical point at infinity $(y_{1}, \cdots, y_{p})_{\infty}$ is at level
$$
C_{\infty}(y_{1}, \cdots, y_{p}):=S_{n}^{\frac{2 \sigma}{n}}\Big(\sum_{i=1}^{p} \frac{1}{K(y_{{i}})^{\frac{n-2 \sigma}{2 \sigma}}}\Big)^{\frac{2 \sigma}{n}}.
$$
\end{corollary}

At the end of this section, we derive the topological contribution of the critical points at infinity to the difference of topology between the level sets of the functional $J_{K}$. As a consequence of Corollary \ref{cor:3.6}, Proposition \ref{pro:3.4}, and the Morse reduction in Lemma \ref{lem:3.5}, we have

\begin{lemma}\label{lem:3.7}
Let $\tau_{\infty}$ be a critical point at infinity at the level $C_{\infty}(\tau_{\infty})$ with index $i_{\infty}(\tau_{\infty}) .$ Then for $\theta$ being a small positive number and a field $G$, we have
$$
H_{q}(J_{K}^{C_{\infty}(\tau_{\infty})+\theta}, J_{K}^{C_{\infty}(\tau_{\infty})-\theta} ; G)=
\begin{cases}
G & \text { if } \, q=i_{\infty}(\tau_{\infty}), \\
0 & \text { otherwise},
\end{cases}
$$
where $J_{K}^{A}:=\{u\in H^{\sigma}(\mathbb{S}^n)\mid J_{K}(u)\leq A\}$ and $H_{q}$ denotes the $q$-dimensional homology group with coefficients in the field $G$.
\end{lemma}

\section{Proofs of Theorem \ref{thm:1.2} and Theorem \ref{thm:1.4}}
This section is devoted to the proof of Theorems \ref{thm:1.2} and \ref{thm:1.4}. Our proofs are based on the characterization of the critical points at infinity in Proposition \ref{pro:3.4} and the computation of their contribution to the difference of topology between level sets in Lemma \ref{lem:3.7}. In addition, the proof of these two theorems needs two deformation lemmas. The first one reads as follows:

\begin{lemma}\label{lem:4.1}
Let $\underline{A}>0$ be a constant and $\overline{A}:=(K_{\max } / K_{\min })^{(n-2\sigma) /n} \underline{A}$. Assume that $J_{K}$ does not have any critical point nor critical point at infinity in the set $J_{K}^{\overline{A}} \backslash J_{K}^{\underline{A}}$. Then for each $c \in [\underline{A}, \overline{A}]$, the level set $J_{K}^{c}$ is contractible.
\end{lemma}

\begin{proof}
Since we assumed that $J_{K}$ does not have any critical point nor critical point at infinity in $\Sigma^{+}$ between the levels $\underline{A}$ and $\overline{A}$, we have that $J_{K}^{\overline{A}}$ retracts by deformation onto $J_{K}^{\underline{A}}$. Indeed, such a retraction can be realized by following the flow lines of a decreasing pseudogradient $Z$ for $J_{K}$. In the sequel, we denote by $\phi_{K}$ the one parameter group corresponding to this pseudogradient. For each $u \in \Sigma^{+}$, we use $s_{K}(u)$ to denote the first time such that $\phi_{K}(s_{K}(u), u) \in J_{K}^{\underline{A}}$.

Recall that the only critical points of $J_{1}$ are minimum point and lie in the bottom level $S_{n}$, where $J_1:=J_{K\equiv1}$. Furthermore, by following the flow lines of a decreasing pseudogradient $Z_{1}$ of the functional $J_{1}$, each flow line, starting from $u \in \Sigma^{+}$, will reach the bottom level $S_{n}$. Hence the set $J_{1}^{A}$ is a contractible one for each $A>S_{n}$. We denote by $\phi_{1}$ the one parameter group corresponding to pseudogradient $Z_{1}$.

Since
$$
(1 / K_{\max }^{(n-2\sigma) /n}) J_{1}(u) \leq J_{K}(u) \leq(1 / K_{\min }^{(n-2\sigma) /n}) J_{1}(u)\quad \text { for each }\, u \in \Sigma^{+},
$$
we get
$$
J_{K}^{\underline{A}} \subset J_{1}^{A^{\prime}} \subset J_{K}^{\overline{A}},
$$
where  $A^{\prime}:=K_{\max }^{(n-2\sigma) /n} \underline{A}$. Moreover, we observe that for each $u \in \Sigma^{+}$, there exists a unique $s_{1}(u)$ satisfying $\phi_{1}(s_{1}(u), u) \in J_{1}^{A^{\prime}}$.

Define
$$
\begin{aligned}
F: [0,1] \times J_{1}^{A^{\prime}} &\rightarrow J_{1}^{A^{\prime}}\\
(t, u) &\mapsto \phi_{1}(s_{1}(\phi_{K}(t s_{K}(u), u)), \phi_{K}(t s_{K}(u), u)).
\end{aligned}
$$
It is easy to see that $F$ is well defined and continuous. Furthermore, $F$ satisfies the following properties:
\begin{itemize}
  \item For $t=0$, $\phi_{K}(0, u)=u$. Moreover, for each $u \in J_{1}^{A^{\prime}}$, $s_{1}(u)=0$. Therefore, for each $u \in J_{1}^{A^{\prime}}$, $F(0, u)=\phi_{1}(0, u)=u$.
  \item For $t=1$, by the definition of $s_{K}$, we have $\phi_{K}(s_{K}(u), u) \in J_{K}^{\underline{A}} \subset J_{1}^{A^{\prime}}$, which implies that $s_{1}(\phi_{K}(s_{K}(u), u))=0$. Therefore, for each $u \in J_{1}^{A^{\prime}}$, $F(1, u)=\phi_{1}(0, \phi_{K}(s_{K}(u), u))=\phi_{K}(s_{K}(u), u) \in J_{K}^{\underline{A}}$.
  \item If $u \in J_{K}^{\underline{A}}$, then $s_{K}(u)=0$, which implies that $\phi_{K}(t s_{K}(u), u)=\phi_{K}(0, u)=u$. Therefore, for each $u \in J_{K}^{\underline{A}}$ and each $t \in[0,1]$, $F(t, u)=\phi_{1}(s_{1}(u), u)=\phi_{1}(0, u)=u$. Notice that $s_{1}(u)=0$ since $u \in J_{K}^{\underline{A}} \subset J_{1}^{A^{\prime}}$.
\end{itemize}
Thus $J_{1}^{A^{\prime}}$ retracts by deformation onto $J_{K}^{\underline{A}}$. Since $J_{1}^{A^{\prime}}$ is a contractible set, this finishes the proof.
\end{proof}

For $\ell \in \mathbb{N}$, define
$$
C_{\max }^{\ell, \infty}:=(\ell S_{n})^{2\sigma/n} / K_{\min }^{(n-2\sigma) /n} \quad \text { and } \quad C_{\min }^{\ell, \infty}:=(\ell S_{n})^{2\sigma/n} / K_{\max }^{(n-2\sigma) /n}.
$$
Using Corollary \ref{cor:3.6}, it is easy to see that the level of critical points at infinity corresponding to $\ell$ points lies between $C_{\min }^{\ell, \infty}$ and $C_{\max }^{\ell, \infty}$. Indeed, by Corollary \ref{cor:3.6} we have
$$
(p S_{n})^{2\sigma/n} / K_{\max }^{(n-2\sigma) /n}\leq C_{\infty}(y_{1}, \cdots, y_{p})\leq (p S_{n})^{2\sigma/n} / K_{\min }^{(n-2\sigma) /n}.
$$

The second deformation lemma is a consequence of Lemma \ref{lem:4.1} and an appropriate pinching condition imposed to the function $K$.

\begin{proposition}\label{pro:4.2}
For $k \in \mathbb{N}$ being fixed, let $K$ satisfying the condition ${\bf (nd)}$ and the pinching condition $K_{\max } / K_{\min }<((k+1) / k)^{\sigma /(n-2\sigma)}$. If $J_{K}$ does not have any critical point under the level $C_{\min }^{k+1, \infty}$. Then, for every $1 \leq \ell \leq k$ and every $c \in(C_{\max }^{\ell, \infty}, C_{\min }^{\ell+1, \infty})$, the level set $J_{K}^{c}$ is contractible.
\end{proposition}

\begin{proof}
Using $K_{\max } / K_{\min }<((k+1) / k)^{\sigma /(n-2\sigma)}$, we get for each $1 \leq \ell \leq k$, $(k+1) / k \leq (\ell+1) / \ell$ and
$$
C_{\max }^{\ell, \infty}<C_{\max }^{\ell, \infty}(K_{\max } / K_{\min })^{(n-2\sigma) /n}<C_{\min }^{\ell+1, \infty}.
$$
Indeed, we have
$$
\begin{aligned}
C_{\max }^{\ell, \infty} &<C_{\max }^{\ell, \infty}(K_{\max } / K_{\min })^{(n-2\sigma) /n} \\
&<(\ell S_{n})^{2\sigma /n} / K_{\min }^{(n-2\sigma) /n}((\ell+1) / \ell)^{\sigma/n} \\
&=((\ell+1) S_{n})^{2\sigma /n} / K_{\min }^{(n-2\sigma) /n}(\ell /(\ell+1))^{\sigma/n} \\
&<((\ell+1) S_{n})^{2\sigma /n} / K_{\min }^{(n-2\sigma) /n}(K_{\max } / K_{\min })^{-(n-2\sigma) /n} \\
&<C_{\min }^{\ell+1, \infty}.
\end{aligned}
$$
By taking $\underline{A}=C_{\max }^{\ell, \infty}+\gamma$ with $\gamma>0$ small enough such that $\overline{A}<C_{\min }^{\ell+1, \infty}$, it is easy to see that the functional $J_{K}$ does not have any critical point nor critical point at infinity between the levels $\underline{A}$ and $\overline{A}$. Hence this proposition follows from Lemma \ref{lem:4.1}.
\end{proof}

Now we are ready to complete the proofs of the two main theorems.

\begin{proof}[Proof of Theorem \ref{thm:1.4}]
Suppose the contrary, then the functional $J_{K}$ does not have any critical point. In particular, $J_{K}$ has no critical points under the level $C_{\min }^{2, \infty}$. Under the assumption of Theorem \ref{thm:1.4}, it follows from Proposition \ref{pro:4.2} (with $k=1$) that $J_K^{C_{\max }^{1, \infty}+\gamma}$ is a contractible set for $\gamma>0$ small enough, and it is a retract by deformation of $J^{C_{\min}^{2, \infty}}_K$.

Using Proposition \ref{pro:3.4} and Corollary \ref{cor:3.6}, we obtain that critical points at infinity under the level $C_{\min }^{2, \infty}$ are in one-to-one correspondence with critical points of $K$ in $\mathcal{K}^+$. Then it follows from Lemma \ref{lem:3.7} and the Euler-Poincar\'e theorem that
$$
1=\chi(J_K^{C_{\min }^{2, \infty}+\gamma})=\sum_{z \in \mathcal{K}^+}(-1)^{n-i n d(K, z)},
$$
which contradicts to the assumption (ii) of Theorem \ref{thm:1.4}. This finishes the proof.
\end{proof}

\begin{proof}[Proof of Theorem \ref{thm:1.2}]
Notice that under the assumption of this theorem, if $A_{1} \neq 1$, the existence of at least one solution to \eqref{1.1} follows from Theorem \ref{thm:1.4}. Therefore, we only need to consider the case $A_{1}=1$.

We claim that the number $N:=\sharp \mathcal{K}^+$ has to be odd. In fact, if $A_{1}=1$ we get $p-q=1$, where $p$ and $q$ are the number of critical points in $\mathcal{K}^+$ with even and odd number $\iota(z)$ respectively, where $\iota(z):=n-{ind}(K, z)$. On the other hand, we have $p+q=N$. Thus we obtain $2 p=N+1$. This completes the proof of our claim.

Assuming that $J_K$ does not have any critical point. In particular, $J_{K}$ has no critical points under the level $C_{\min }^{3, \infty}$. Applying Proposition \ref{pro:4.2} with $k=2$, we obtain that the level sets $J_K^{C_{\max }^{1, \infty}+\gamma}$ and $J_K^{C_{\max }^{2, \infty}+\gamma}$ are contractible sets for $\gamma>0$ small enough. Then it follows from the properties of the Euler characteristic (see \cite[Proposition 5.7]{DoldLectures1995}) that
$$
1=\chi(J_K^{C_{\max }^{2, \infty}+\gamma})=\chi(J_K^{C_{\max }^{2, \infty}+\gamma}, J_K^{C_{\max }^{1, \infty}+\gamma})+\chi(J_K^{C_{\max }^{1, \infty}+\gamma})=\chi(J_K^{C_{\max }^{2, \infty}+\gamma}, J_K^{C_{\max }^{1, \infty}+\gamma})+1.
$$
That is $\chi(J_K^{C_{\max }^{2, \infty}+\gamma}, J_K^{C_{\max }^{1, \infty}+\gamma})=0 .$ Furthermore, it follows from Proposition \ref{pro:3.4} and Corollary \ref{cor:3.6} that the critical points at infinity between these two levels are $(y_{i}, y_{j})_{\infty}$ with $y_{i} \neq y_{j} \in \mathcal{K}^+$. Thus, it follows from Lemma \ref{lem:3.7} and the Euler-Poincar\'e theorem that
\begin{equation}\label{4.1}
\sum_{y_{i} \neq y_{j} \in \mathcal{K}^+}(-1)^{1+\iota(y_{i})+\iota(y_{j})}=0.
\end{equation}
Recall that the number $N:=\sharp \mathcal{K}^+$ is odd, i.e., $N=2 k+1$ for some $k \in \mathbb{N}$, and there are $k$ odd numbers $\iota(y_{i})$'s and $k+1$ even numbers $\iota(y_{i})$'s. We claim that
\begin{equation}\label{4.2}
A_{2}:=\sum_{i<j}(-1)^{\iota(y_{i})+\iota(y_{j})}=-k.
\end{equation}
In fact, notice that the value of $A_{2}$ is the sum of $+1$ and $-1$. To get +1, $\iota(y_{i})$ and $\iota(y_{j})$ have to be of the same parity, and to get $-1$, $\iota(y_{i})$ and $\iota(y_{j})$ have to be of different parity. Therefore
\begin{itemize}
  \item For $k=0$, we have only one point $y$ with an even $\iota(y)$. Thus $A_{2}=0$.
  \item For $k=1$, we have two points $y_{0}$ and $y_{2}$ with even $\iota(y_{i})$ and one point $y_{1}$ with an odd $\iota(y_{1})$. Thus, $A_{2}=1-2=-1$.
  \item For $k \geq 2$, we have $k+1$ points with even $\iota(y_{i})$ and $k$ points with odd $\iota(y_{i})$. Thus,
$$
A_2=C_{k+1}^2+C_k^2-C_{k+1}^1C_k^1=\frac{1}{2}(k+1) k+\frac{1}{2} k(k-1)-(k+1) k=-k.
$$
\end{itemize}
Then the claim follows from the above arguments. Now, \eqref{4.1} and \eqref{4.2} imply that $k=0$ and hence $\sharp \mathcal{K}^+=1$. This contradicts to our assumption $\sharp \mathcal{K}^+ \geq 2$. The proof of Theorem \ref{thm:1.2} is complete.
\end{proof}

\noindent Z. Tang \& N. Zhou

\noindent School of Mathematical Sciences, Laboratory of Mathematics and Complex Systems, MOE, Beijing Normal University, Beijing, 100875, China\\[1mm]
Email: \textsf{tangzw@bnu.edu.cn}\\
Email: \textsf{nzhou@mail.bnu.edu.cn}

\end{document}